\documentclass[reqno]{amsart}

\usepackage[latin1]{inputenc} 
\usepackage{latexsym,amsfonts,amssymb,amsmath,euscript}

\usepackage{graphicx} 
\usepackage{color,fancybox}
\usepackage{color}
\usepackage{graphicx} 

\newcommand{\R}
 {\mathbb{R}}

\begin{document}

\title[Terms concentrating on boundary]{A nonlinear elliptic problem with terms \\ concentrating  in the boundary}

\author[G. S. Arag\~ao, A. L. Pereira and M. C. Pereira]
{Gleiciane S. Arag\~ao$^\star$, Ant\^onio L. Pereira$^\diamond$ and Marcone C. Pereira$^\dagger$}

\thanks{$^\star$Partially supported by FAPESP 2010/51829-7, Brazil, $^\diamond$Partially
supported by CNPq 308696/2006-9, FAPESP 2008/55516-3, Brazil, $^\dagger$Partially
supported by CNPq 305210/2008-4, FAPESP 2008/53094-4, 2010/18790-0 and 2011/08929-3, Brazil}

\address{Gleiciane S. Arag\~ao \hfill\break
Universidade Federal de S\~ao Paulo - Diadema - Brazil}
\email{gleiciane.aragao@unifesp.br}

\address{Ant\^onio L. Pereira \hfill\break
Instituto de Matem\'atica e Estat\'istica \\
Universidade de S\~ao Paulo - S\~ao Paulo - Brazil}
\email{alpereir@ime.usp.br}

\address{Marcone C. Pereira \hfill\break
Escola de Artes, Ci\^encias e Humanidades\\
Universidade de S\~ao Paulo - S\~ao Paulo - Brazil}
\email{marcone@usp.br}

\date{}


\begin{abstract}

In this paper we investigate the behavior of a family of steady state solutions of a nonlinear reaction diffusion equation when some reaction and potential terms are concentrated in a $\epsilon$-neighborhood of a portion $\Gamma$ of the boundary. We assume that this $\epsilon$-neighborhood shrinks to $\Gamma$ as the small parameter $\epsilon$ goes to zero. Also, we suppose the upper boundary of this $\epsilon$-strip presents a highly oscillatory behavior. Our main goal here is to show that this family of solutions converges to the solutions of a limit problem, a nonlinear elliptic equation that captures the oscillatory behavior. Indeed, the reaction term and concentrating potential are transformed into a flux condition and a potential on $\Gamma$, which depends on the oscillating neighborhood.

\end{abstract}

\maketitle
\numberwithin{equation}{section}
\newtheorem{theorem}{Theorem}[section]
\newtheorem{lemma}[theorem]{Lemma}
\newtheorem{corollary}[theorem]{Corollary}
\newtheorem{proposition}[theorem]{Proposition}
\newtheorem{remark}[theorem]{Remark}
\allowdisplaybreaks

\section{Introduction}

In this work we analyze the behavior of a family of steady state solutions of a homogeneous Neumann problem for a nonlinear reaction diffusion equation when some reaction and potential terms are concentrated in a $\epsilon$-neighborhood of a subset $\Gamma$ of the boundary that shrinks to $\Gamma$ as the small parameter $\epsilon$ goes to zero.
Roughly, we are assuming that some reactions of the system occur only in an extremely thin region near the border with oscillating upper boundary. 
We show that in some sense this singular problem can be approximated by a nonlinear elliptic system with nonlinear boundary conditions where the oscillatory behavior of the neighborhood is captured as a flux condition and a potential term on the portion $\Gamma$ of the boundary.

To describe the problem, let $\Omega = (0,1) \times (0,1)$ be the open square in $\R^2$ and let $\Gamma \subset \partial \Omega$ be the line segment in $\R^2$ given by
$$
\Gamma = \{ (x,0) \in \R^2 \; : \; x \in (0,1) \}.
$$ 
We consider the following uniformly bounded $\epsilon$-neighborhood of $\Gamma$ with oscillatory upper boundary
$$
\omega_\epsilon = \{ (x,y) \in \R^2 \; : \; x \in (0,1) \quad \mbox{and} \quad 0 < y < \epsilon \, G_\epsilon (x) \}.
$$ 
Here we assume that $G_\epsilon(\cdot)$ is a function satisfying $0 < G_0 \leq G_\epsilon(\cdot) \leq G_1$ for fixed positive constants $G_0$ and $G_1$ which oscillates as the small parameter $\epsilon \to 0$. 
This is expressed by assuming that 
\begin{equation} \label{defG}
G_\epsilon(x) =G(x,x/\epsilon).
\end{equation}
The function $G : (0, 1) \times \R \mapsto \R$ is a positive smooth function, with $y \to G(x, y)$ periodic in $y$ for fixed $x$ with period $l(x)$ uniformly bounded in $(0,1)$, that is, $0<l_0<l(\cdot)<l_1$.
Let us observe that our assumptions includes 
the case where the function $G_\epsilon$ presents a purely periodic behavior, for instance, $G_\epsilon(x) = 2 + cos(x/\epsilon)$. 
But it also considers 
the case where the function $G_\epsilon$ defines a strip where the oscillations period, the amplitude and the profile vary with respect to $x \in (0, 1)$. 
See Figure \ref{fig1} below that illustrates the oscillating strip $\omega_\epsilon \subset \Omega$.

\begin{figure}[htp] 
\centering \scalebox{0.6}{\includegraphics{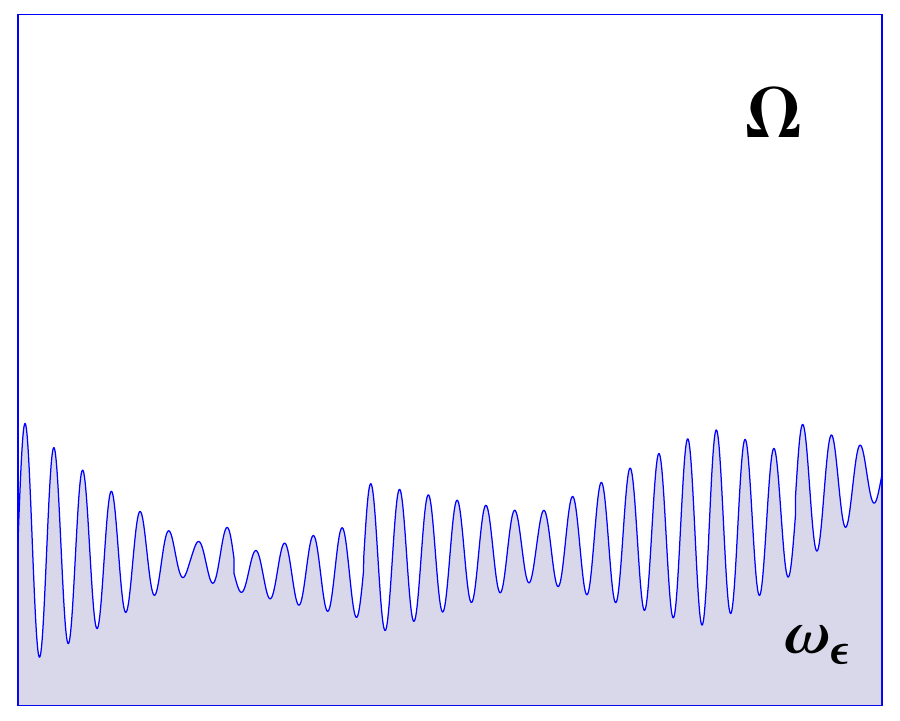}}
\caption{The open square $\Omega$ and the $\epsilon$-strip $\omega_\epsilon$.}
\label{fig1} 
\end{figure}

We are interested in  the behavior of the solutions of the nonlinear elliptic problem
\begin{eqnarray}
\label{1}
\left\{ 
\begin{gathered}
- \Delta u^{\epsilon}  + \lambda u^{\epsilon} + \frac{1}{\epsilon} \mathcal{X}_{\omega_{\epsilon}} \, V_\epsilon(\cdot) \, u^{\epsilon} 
= \frac{1}{\epsilon} \mathcal{X}_{\omega_{\epsilon}} \, f_0(\cdot,u^\epsilon) + f_1(\cdot,u^\epsilon) \quad \textrm{ in }  \Omega \\
     \frac{\partial u^{\epsilon}}{\partial n}=0 \quad \textrm{ on } \partial \Omega
        \end{gathered} \right.
\end{eqnarray}
where $\mathcal{X}_{\omega_{\epsilon}}$ is the characteristic function of the set $\omega_{\epsilon}$, $n$ denotes the unit outward normal vector to $\partial \Omega$ and $\lambda >0$ is a suitable real number. 
The nonlinearities $f_0, f_1: \mathcal{O} \times \R \mapsto \R$ are continuous in both variables and $\mathcal{C}^2$ in the second one,  where $\mathcal{O} \subset \R^2$ is an open set containing $\bar \Omega$. 
We also assume that there exists $C>0$ independent of $\epsilon$ such that the family of potential $V_{\epsilon} \in L^\infty(\Omega)$ satisfies
\begin{eqnarray}
\label{hip1}
\frac{1}{\epsilon} \int_{\omega_{\epsilon}} \left|V_{\epsilon}(x,y) \right|^{2} \, dx dy \leq C.
\end{eqnarray} 
Also, we suppose there exists a function $V_{0}\in L^{2}(\Gamma)$ which is the limit of the concentrating term 
\begin{eqnarray}
\label{hip2}
\lim_{\epsilon\to 0} \frac{1}{\epsilon} \int_{\omega_{\epsilon}} V_{\epsilon} \, \varphi \, d\xi = \int_{\Gamma}V_{0} \, \varphi \, dS, \quad \mbox{$\forall \varphi \in C^{\infty}(\bar{\Omega})$}.
\end{eqnarray}    

In our model, we use the characteristic function $\mathcal{X}_{\omega_{\epsilon}}$ and the small positive parameter $\epsilon$ to express the concentration on the region $\omega_\epsilon \subset \Omega$ by the term 
$$
\frac{1}{\epsilon} \mathcal{X}_{\omega_{\epsilon}} \in L^\infty(\Omega).
$$
Since $\omega_\epsilon \subset (0,1) \times (0, \epsilon \, G_1)$ is thin and it is ``approaching'' the line segment $\Gamma \subset \partial \Omega$, it is reasonable to expect that the family of solutions $u^\epsilon$ will converge to a
solution of  an equation of the same type with nonlinear boundary condition on $\Gamma$.
Indeed, we will show that, for $\lambda$ big enough, the solutions of (\ref{1}) converge in $H^{1}(\Omega)$ to the solutions of the nonlinear elliptic problem
\begin{equation}
\label{2}
\left\{ 
\begin{gathered}
       -\Delta u + \lambda u = f_1(\cdot,u) \quad \textrm{ in } \Omega \\
       \frac{\partial u}{\partial n} + V_{0}(\cdot) u = \mu(\cdot) \, f_0(\cdot,u)  \quad \textrm{ on } \Gamma \\
       \frac{\partial u}{\partial n} = 0  \quad \textrm{ on } \partial\Omega \setminus \Gamma\\
        \end{gathered} \right.
\end{equation}
where the boundary coefficient $\mu \in L^\infty( \Gamma)$ is related to the oscillating function $G_\epsilon$ and is given by 
\begin{equation} \label{mu}
\mu(x) = \frac{1}{l(x)} \int_0^{l(x)} G(x,y) \, dy, \quad \forall x \in (0,1).
\end{equation}

As previously mentioned, we obtain a limit problem with nonlinear boundary condition that captures the oscillatory behavior of the upper boundary of the set $\omega_{\epsilon}$. In fact, its nonlinear boundary condition involves the function $\mu(x)$, that is, the mean value of $G(x,\cdot)$ for each $x \in (0,1)$.
 We now summarize these assertions precisely as our  main result.
\begin{theorem} \label{main}
Let $u^\epsilon$ be a family of solutions of the problem \eqref{1} satisfying $\| u^\epsilon \|_{L^\infty(\Omega)} \leq R$ for some positive constant $R$ independent of $\epsilon$.
Then, there exists $\lambda^* \in \R$ independent of $\epsilon$ such that, for all $\lambda > \lambda^*$,  there exists a subsequence, still defined by $u^\epsilon$, and a function $u \in H^1(\Omega)$, with $\|u\|_{L^\infty(\Omega)} \leq R$, solution of the problem \eqref{2} satisfying $\| u^\epsilon - u\|_{H^1(\Omega)} \to 0$ as $\epsilon \to 0$.
%
\end{theorem}

Since we are concerned with solutions which are  uniformly bounded  in $L^\infty(\Omega)$, we may assume that the nonlinearities $f_0$ and $f_1$ are $\mathcal{C}^2$-functions with bounded derivatives in the second variable.
Indeed, we may perform a cut-off in $f_0$ and $f_1$ outside the region 
 $|u| \leq R$,  without modifying any of  these solutions.

This kind of problem was initially studied in~\cite{arrieta} where linear elliptic equations posed on $\mathcal{C}^2$-regular domains were considered. There, the $\epsilon$-neighborhood is a strip of width $\epsilon$ with base in a portion of the boundary without oscillatory behavior. Later, the asymptotic behavior of the attractors of a parabolic problem were analyzed in~\cite{anibal,anibal2}, where the upper semicontinuity of attractors at $\epsilon=0$ was proved. The same technique of~\cite{arrieta} has been used in~\cite{gleicesergio1,gleicesergio2}, where the results of~\cite{arrieta,anibal} were extended to reaction-diffusion problems with delay.
 
The goal of our work is to extend some results of~\cite{arrieta} to nonlinear elliptic problems when the upper boundary of the strip $\omega_{\epsilon}$ presents a highly oscillatory behavior. We also use some ideas of~\cite{marcone1,marcone2,marcone3} where elliptic and parabolic problems defined in thin domains with a highly oscillatory boundary have been extensively studied. It is important to note that in our work the boundary $\partial \Omega$ of the domain $\Omega$ is only Lipschitz . 


\section{Some technical results}
\label{tecnicos}
  
In this section we describe some technical results that will be needed in the proof of the main result.
We  initially analyze how our concentrating integrals converge to boundary integrals. 
We adapt the results of \cite{arrieta} on convergence of concentrated integrals using that  $0<G_{0}\leq G_\epsilon(\cdot)\leq G_{1}$ uniformly in $\epsilon >0$.

\begin{lemma}
\label{lem1}
Suppose that $v\in H^{s}(\Omega)$ with $1/2<s\leq 1$ and $s-1\geq - 1/q$. Then, for small $\epsilon_{0}$, there exists a constant $C>0$ independent of $\epsilon$ and $v$ such that for any $0<\epsilon\leq\epsilon_{0}$, we have
$$
\frac{1}{\epsilon}\int_{\omega_{\epsilon}}\left|v\right|^{q} \, d\xi \leq C \left\|v\right\|^{q}_{H^{s}(\Omega)}.
$$
\end{lemma}
\begin{proof}
Note that
$$
\frac{1}{\epsilon}\int_{\omega_{\epsilon}}\left|v\right|^q d\xi\leq \frac{1}{\epsilon}\int_{r_{\epsilon}}\left|v\right|^q d\xi,
$$
where $r_{\epsilon}$ is the strip of width $\epsilon \, G_{1}$ and base in $\Gamma$ without oscillatory behavior given by
$$
r_{\epsilon}=\left\{(x,y)\in \mathbb{R}^{2} \; : \;  x\in (0,1) \quad \mbox{and} \quad 0< y<\epsilon G_{1}\right\}. 
$$
Since $\omega_{\epsilon}$ is contained in $r_{\epsilon}$, the result follows from ~\cite[Lemma 2.1]{arrieta}.
\end{proof}

As a consequence of Lemma \ref{lem1} and \cite[Lemma 2.5]{arrieta} we  obtain:
\begin{lemma} \label{Potential}
Suppose that the family $V_{\epsilon}$ satisfies (\ref{hip1}) and (\ref{hip2}). Then, for $s>1/2$, $\sigma>1/2$ and $s+\sigma>3/2$, if we define the operators $T_{\epsilon}: H^{s}(\Omega) \rightarrow (H^{\sigma}(\Omega))'$ by
$$
\left\langle T_{\epsilon}(u),\varphi\right\rangle=\frac{1}{\epsilon}\int_{\omega_{\epsilon}}V_{\epsilon} \, u \, \varphi \, d\xi, \quad \textrm { for }  \epsilon>0,   \quad \mbox{ and } \quad  \left\langle T_{0}(u),\varphi\right\rangle=\int_{\Gamma} V_{0} \, u \, \varphi \, dS,
$$
we have $T_{\epsilon} \to T_{0}$ in $\mathcal{L}(H^{s}(\Omega),(H^{\sigma}(\Omega))')$.
\end{lemma}

We also need some results about weak limits of rapidly oscillating functions.
\begin{lemma} \label{LGE}
If $G_{\epsilon}$  is defined as in (\ref{defG}), then
$$
G_\epsilon(\cdot)  \to  \mu(\cdot) = \frac{1}{l(\cdot)}\int_0^{l(\cdot)}G(\cdot,s) \,ds \quad  w^*-L^\infty(0,1).
$$
\end{lemma}
\begin{proof}
We have to prove
$$
\int_0^1 \left\{
G_\epsilon(x) - \frac{1}{l(x)} \int_0^{l(x)} G(x,s) \, ds
\right\} \varphi(x) \, dx \to 0 \textrm{ as } \epsilon \to 0, \forall \varphi \in L^1(0,1).
$$ 
Since the set of step function is dense in $L^1(0,1)$ and any step function is a linear combination of characteristic functions, it is enough to take   
$
\varphi= \mathcal{X}_{(e,f)},
$
the characteristic function on the interval $(e,f)$, 
for $0 \leq e < f \leq 1$.
So, we have to estimate the integral
$$
I_{e,f} = \int_e^f \left\{
G_\epsilon(x) - \frac{1}{l(x)} \int_0^{l(x)} G(x,s) \, ds
\right\} dx.
$$
For this, let $\eta>0$ be a small number and let $\{ e=x_0, x_1, ..., x_n=f \}$ be a partition for the interval $(e,f)$, and 
$\hat x_i$ be a fixed point in the interval $J_i=[x_{i-1},x_i]$, $i=1, ..., n$, such that 
$$ 
\sup_i \sup_{x \in J_i, \, y \in \R}
| G(x,y) - G(\hat x_i,y) | < \eta.
$$
Observe that we can write 
$ 
I_{e,f} = \sum_{i=1}^5 I_{e,f}^i,
$
where
$$
\begin{gathered}
I_{e,f}^1 =   
\sum_{i=1}^n \int_{J_i} \left\{ G(x,x/\epsilon) - G(\hat x_i,x/\epsilon) \right\} dx, \\ 
I_{e,f}^2 = \sum_{i=1}^n \int_{J_i} \left\{ G(\hat x_i,x/\epsilon) - 
\frac{1}{l(\hat x_i)} \int_0^{l(\hat x_i)} G(\hat x_i,s) \, ds \right\} dx \\ 
I_{e,f}^3 = \sum_{i=1}^n \int_{J_i} \left\{ \frac{1}{l(\hat x_i)} \int_0^{l(\hat x_i)} G(\hat x_i,s) \, ds - 
\frac{1}{l(\hat x_i)} \int_0^{l(\hat x_i)} G(x,s) \, ds \right\} dx \\
I_{e,f}^4 = \sum_{i=1}^n \int_{J_i} \left\{ \frac{1}{l(\hat x_i)} \int_0^{l(\hat x_i)} G(x,s) \, ds - 
\frac{1}{l(\hat x_i)} \int_0^{l(x)} G(x,s) \, ds \right\} dx \\
I_{e,f}^5 = \sum_{i=1}^n \int_{J_i} \left\{ \frac{1}{l(\hat x_i)} \int_0^{l(x)} G(x,s) \, ds - 
\frac{1}{l(x)} \int_0^{l(x)} G(x,s) \, ds \right\} dx .
\end{gathered} 
$$
It is easy to estimate the integrals $I_{e,f}^1$, $I_{e,f}^3$, $I_{e,f}^4$ and $I_{e,f}^5$ to obtain
$
|I_{e,f}^1| \leq \eta \, (f-e)
$, 
$
|I_{e,f}^3| \leq \eta \, (f-e)
$, 
$
|I_{e,f}^4| \leq G_1 \, \| \hat l^\eta - l \|_{L^\infty(0,1)} \, (f-e)
$ and 
$
|I_{e,f}^5| \leq G_1 \, \left( l_1/{l_0}^2 \right) \, \| \hat l^\eta - l \|_{L^\infty(0,1)} \, (f-e)
$,
where the function $\hat l^\eta$ is the step function defined for each $\eta > 0$ by 
$
\hat l^\eta(x) = l(x_i)  \textrm{ as } x_i \in J_i.
$
Since these inequalities do not depend on $\epsilon > 0$, and
$\| \hat l^\eta - l \|_{L^\infty(0,1)} \to 0$ as $\eta \to 0$ uniformly in $\epsilon$,
we have that $I_{e,f}^1$, $I_{e,f}^3$, $I_{e,f}^4$ and $I_{e,f}^5$ goes to zero as $\eta \to 0$
uniformly in $\epsilon > 0$.
Hence, to conclude the proof, we just evaluate the integral $I_{e,f}^2$.
But this is a application of~\cite[Theorem 2.6]{cioranescu}. 
\end{proof}

The following result will also be needed. 
\begin{lemma}
\label{lem2}
Suppose that $h, \varphi \in H^{s}(\Omega)$ with $1/2 < s \leq 1$. Then, 
\begin{eqnarray} \label{SL2}
\label{limiteexplicito}
\lim_{\epsilon \to 0}\frac{1}{\epsilon}\int_{\omega_{\epsilon}} h \, \varphi \, d\xi = \int_{\Gamma} \mu \, \gamma(h) \, \gamma(\varphi) \, dS, 
\end{eqnarray}
where $\mu \in L^\infty(\Gamma)$ is given by \eqref{mu} and $\gamma: H^{s}(\Omega) \mapsto L^2(\Gamma)$ is the trace operator.
\end{lemma}
\begin{proof}
Initially, let $h$ and $\varphi$ be smooth functions defined in $\bar{\Omega}$ independent of $\epsilon$. 
Note that
$$
\frac{1}{\epsilon}\int_{\omega_{\epsilon}} h \, \varphi \, d\xi -  \int_{\Gamma} \mu \, h \, \varphi \, dS = \frac{1}{\epsilon} \int^{1}_{0} \int^{\epsilon G_\epsilon(x)}_{0} h(x,y) \, \varphi(x,y) \, dy dx - \int^{1}_{0} \mu(x) \, h(x,0) \, \varphi(x,0) \, dx.
$$
Taking $y=\epsilon G_\epsilon(x) \, z$, we obtain by  changing variables
$$
\frac{1}{\epsilon}\int^{1}_{0}\int^{\epsilon G_\epsilon(x)}_{0} h(x,y) \, \varphi(x,y) \, dy dx = \int^{1}_{0} \int^{1}_{0} h\left(x,\epsilon G_\epsilon(x) z \right) \, \varphi \left(x,\epsilon G_\epsilon(x) z \right) G_\epsilon(x) \, dz dx.
$$
Thus, adding and subtracting $\int^{1}_{0} h(x,0) \, \varphi(x,0) \, G_\epsilon(x) \, dx$, we get
\begin{eqnarray*}
\left|\frac{1}{\epsilon}\int_{\omega_{\epsilon}} h \, \varphi \, d\xi - \int_{\Gamma} \mu \, h \, \varphi \, dS \right| 
& \leq &   \left|\int^{1}_{0} h(x,0) \, \varphi(x,0) \, G_\epsilon(x) \, dx -\int^{1}_{0} \mu(x) \, h(x,0) \, \varphi(x,0) \, dx \right| \\
& + & \left|\int^{1}_{0} \int^{1}_{0} G_\epsilon(x) \left[ h\left(x,\epsilon G_\epsilon(x) z \right) \, \varphi\left(x,\epsilon G_\epsilon(x) z \right) - h(x,0) \, \varphi(x,0) \, \right]dz dx\right| .
\end{eqnarray*}

Now, since $\epsilon \, G_\epsilon(x) z \to 0$ as $\epsilon\to0$ uniformly for $(x,z) \in  [0,1]\times[0,1]$, we have 
\begin{equation} \label{Seila}
\left|\int^{1}_{0} \int^{1}_{0} G_\epsilon(x) \left[h\left(x,\epsilon G_\epsilon(x) z \right) \, \varphi\left(x,\epsilon G_\epsilon(x) z \right) -h(x,0) \, \varphi(x,0) \right] \, dz dx\right|\to 0 \qquad \mbox{as $\epsilon\to0$}.
\end{equation}
Hence, we obtain from Lemma \ref{LGE} and \eqref{Seila} that
$$
\left|\frac{1}{\epsilon}\int_{\omega_{\epsilon}} h \, \varphi \, d\xi -\int_{\Gamma} \mu \, h \, \varphi \, dS \right| \to 0 \mbox{ as $\epsilon\to0$}.
$$
Consequently,  the proof of equality \eqref{SL2} follows by density arguments and the continuity of the trace operator $\gamma$ (see \cite{necas}).
\end{proof}

\section{Abstract setting and existence of solutions} 
To write the problems \eqref{1} and \eqref{2}  in an abstract form, we first
 define  the continuous bilinear forms $a_\epsilon: H^1(\Omega) \times H^1(\Omega) \mapsto \R$, $\epsilon \geq 0$, by
\begin{equation} \label{CF}
\begin{gathered}
a_\epsilon(u,v) = \int_\Omega \nabla u \, \nabla v \, dxdy + \lambda \int_\Omega u \, v \, dxdy + \frac{1}{\epsilon} \int_{\omega_{\epsilon}} V_\epsilon \, u \,  v \, dxdy, \quad \textrm { for }  \epsilon>0, \\
a_0(u,v) = \int_\Omega \nabla u \, \nabla v \, dxdy + \lambda \int_\Omega u \, v \, dxdy +  \int_\Gamma V_0 \, u \,  v \, dxdy.
\end{gathered}
\end{equation}

We then   consider the  linear operator $A_\epsilon: H^1 \subset 
H^{-1}(\Omega) \mapsto H^{-1}(\Omega)$  defined  by the relationship
$$
\left\langle A_\epsilon u,v \right\rangle_{-1,1} = a_\epsilon(u,v), \textrm{ for all } v \in H^1(\Omega).
$$
Now, we  can write the problem \eqref{1} as 
$
A_\epsilon u = F_\epsilon(u), 
$
for $\epsilon>0$, where $F_\epsilon: H^1(\Omega) \mapsto H^{-s}(\Omega)$ with $1/2 < s <1$ is defined by
\begin{equation}  \label{FEp} 
\begin{gathered} 
F_\epsilon = F_{0,\epsilon} + F_1, \\
\left\langle F_{0,\epsilon}(u),v \right\rangle = \frac{1}{\epsilon} \int_{\omega_\epsilon}  f_0(\xi,u) \, v \, d\xi  \quad \textrm{ and } \quad 
\left\langle F_1(u),v \right\rangle = \int_\Omega f_1(\xi,u) \, v \, d\xi, \quad \textrm{ for all } v \in H^{s}(\Omega). 
\end{gathered}
\end{equation}

Similarly, we  can write the problem \eqref{2} in an abstract form as 
$
A_0 u = F(u),
$
where  $F: H^1(\Omega) \mapsto H^{-s}(\Omega)$ with $1/2 < s <1$ is defined by
\begin{equation} \label{F0}
\begin{gathered}
F = F_{0} + F_1, \\
\textrm{ $F_1$ is given by \eqref{FEp} \quad and } \quad \left\langle F_{0}(u),v \right\rangle = \int_{\Gamma} \mu \, \gamma(f_0(\xi,u)) \, \gamma(v) \, d\xi,
\quad \textrm{ for all } v \in H^{s}(\Omega),
\end{gathered}
\end{equation}
where $\mu \in L^\infty(\Gamma)$ is given by \eqref{mu} and $\gamma: H^{s}(\Omega) \mapsto L^2(\Gamma)$ is the trace operator.

\begin{lemma} \label{rem1}
There exists $\lambda^* \in \R$ independent of $\epsilon \geq 0$ such that the bilinear form $a_\epsilon$ is uniformly coercive for all $\lambda > \lambda^*$. 
\end{lemma} 
\begin{proof}
Here we will just consider the case $a_\epsilon$ with $\epsilon > 0$. A similar argument gives the result to the bilinear form $a_{0}$.
First we note that for every $\phi\in H^{1}(\Omega)$ we have
\begin{eqnarray}
\label{t1}
a_{\epsilon}(\phi,\phi) \geqslant \left\|\nabla \phi\right\|^{2}_{L^{2}(\Omega)}+\lambda \left\|\phi\right\|^2_{L^{2}(\Omega)} - \frac{1}{\epsilon} \int_{\omega_{\epsilon}}\left(V_{\epsilon}\right)_{-} \, \left|\phi\right|^2 \, d\xi, 
\end{eqnarray}
where $\left(V_{\epsilon}\right)_{-}$ is the negative part of the potential $V_{\epsilon}$ such that $V_{\epsilon} = \left(V_{\epsilon}\right)_{+} - \left(V_{\epsilon}\right)_{-}$. For this negative part we have the following bound
\begin{eqnarray*}
 \frac{1}{\epsilon} \int_{\omega_{\epsilon}} \left(V_{\epsilon}\right)_{-} \, \left|\phi\right|^2 \, d\xi  & \leqslant & \left(\frac{1}{\epsilon}\int_{\omega_{\epsilon}}\left|V_{\epsilon}\right|^2 \, d\xi \right)^{\frac{1}{2}} \left(\frac{1}{\epsilon}\int_{\omega_{\epsilon}}\left|\phi\right|^4 \, d\xi \right)^{\frac{1}{2}}  
 \leqslant  C \left(\frac{1}{\epsilon}\int_{\omega_{\epsilon}}\left|\phi\right|^4 d\xi \right)^{\frac{1}{2}}.
\end{eqnarray*}
Choosing $\frac{1}{2}<s<1$ and $s-1\geqslant -\frac{1}{4}$, that is, $\frac{3}{4}\leqslant s<1$, and using the Lemma \ref{lem1} with $q=4$, we get 
$$
\frac{1}{\epsilon}\int_{\omega_{\epsilon}}\left(V_{\epsilon}\right)_{-}\left|\phi\right|^2 \, d\xi  \leqslant C\left\|\phi\right\|^{2}_{H^{s}(\Omega)}\leqslant C \left\|\phi\right\|^{2s}_{H^{1}(\Omega)} \left\|\phi\right\|^{2(1-s)}_{L^{2}(\Omega)}.
$$
Next we can use Young's inequality to obtain 
\begin{eqnarray}
\label{t2}
\frac{1}{\epsilon}\int_{\omega_{\epsilon}}\left(V_{\epsilon}\right)_{-}\left|\phi\right|^2 d\xi  \leqslant \delta \left\|\phi\right\|^{2}_{H^{1}(\Omega)} + C_{\delta}\left\|\phi\right\|^{2}_{L^{2}(\Omega)},
\end{eqnarray}
for any $\delta>0$.
Then, it follows from (\ref{t1}) and (\ref{t2}) that
$$
a_{\epsilon}(\phi,\phi) \geqslant  \left( \lambda-(1+C_{\delta})\right)\left\|\phi\right\|^2_{L^{2}(\Omega)}+(1-\delta)\left\|\phi\right\|^{2}_{H^{1}(\Omega)}.
$$ 
Consequently, we can take $\delta>0$ small enough and $\lambda>0$ large enough such that
\begin{eqnarray}
\label{coerciva}
a_{\epsilon}(\phi,\phi)\geqslant C \left\|\phi\right\|^{2}_{H^{1}(\Omega)}, \qquad \mbox{$\forall$ $\phi\in H^{1}(\Omega)$}, 
\end{eqnarray}
with $C=C(\lambda)>0$ independent of $\epsilon$. Therefore, the bilinear form $a_{\epsilon}$, $0<\epsilon\leqslant\epsilon_{0}$, is strictly coercive.
We still note that if $V_{\epsilon}\geqslant0$ in (\ref{CF}), then we can take any $\lambda>0$.
\end{proof}

The Lemma \ref{rem1} implies that $u^\epsilon$ is a solution of \eqref{1} if only if $u^\epsilon \in H^1(\Omega)$ satisfies $u^\epsilon = A_\epsilon^{-1} F_\epsilon (u^\epsilon)$, that is, $u^\epsilon$ must be a fixed point of the nonlinear map 
$$A_\epsilon^{-1} \circ F_\epsilon: H^1(\Omega) \mapsto H^1(\Omega),$$ 
for all $\lambda > \lambda^*$. The existence of solutions of \eqref{1} follows then  from Schauder's Fixed Point Theorem.  
In a very similar way,  the solutions of the limit problem \eqref{2} can be obtained  as fixed points of the map 
$$A_0^{-1} \circ F: H^1(\Omega) \mapsto H^1(\Omega).$$ 
Note that they can also be obtained as limits of solutions of \eqref{1}, as shown in the next section.


\section{Upper semicontinuity of steady states} 

In order to obtain the upper semicontinuity of the family $u^\epsilon$, we study the behavior of the  maps $F_{\epsilon}$ and $F$ defined  in \eqref{FEp} and 
\eqref{F0}.


\begin{lemma}
\label{lem1eq}
$(1)$ If $u\in H^{1}(\Omega)$ satisfies $\| u \|_{L^\infty(\Omega)} \leq R$, then there exists $K>0$ independent of $\epsilon$ such that
$$
\sup_u \left\{ 
\left\| F(u) \right\|_{H^{-s}(\Omega)},
\left\|F_{\epsilon}(u)\right\|_{H^{-s}(\Omega)}
\right\}
\leq K, \quad \textrm{ for all } 1/2<s\leq1.
$$

$(2)$ Suppose $\| u \|_{H^1(\Omega) \cap L^\infty(\Omega)} \leq R$. Then, we have for all $1/2<s<1$
$$
\left\|F_{\epsilon}(u)-F(u)\right\|_{H^{-s}(\Omega)}\to 0 \textrm{ as } \epsilon \to 0, \textrm{ uniformly in } u.
$$

$(3)$ Also, if $u^{\epsilon}\to u$ in $H^{1}(\Omega)$, then
$
\left\|F_{\epsilon}(u^{\epsilon})-F(u)\right\|_{H^{-s}(\Omega)}\to 0 \mbox{ as $\epsilon\to0$}.
$
\end{lemma}
\begin{proof}
We can get (1) from continuity of the nonlinearities and Lemma \ref{lem1}. Part (3) follows from (2) adding and subtracting $F_\epsilon(u)$. So, we just have to prove (2). For this, it is enough consider the maps $F_{0,\epsilon}$ and $F_0$ given by \eqref{FEp} and \eqref{F0}. Initially, take $s_{0}$ satisfying $1/2< s_{0} < 1$. For each $u\in H^{1}(\Omega)$ and $\phi\in H^{s_{0}}(\Omega)$
$$
\left| \left\langle F_{0,\epsilon}(u),\phi\right\rangle-\left\langle F_{0}(u),\phi\right\rangle \right| = \left|\frac{1}{\epsilon}\int_{\omega_{\epsilon}} f_0(\xi,u(\xi)) \, \phi(\xi) \, d\xi - \int_{\Gamma} \mu(\xi) \, \gamma\left(f_0(\xi, u(\xi))\right) \, \gamma\left(\phi(\xi)\right) \, d\xi \right|.
$$
From Lemma~\ref{lem2}, we have 
$
\left\langle F_{0,\epsilon}(u),\phi\right\rangle \to \left\langle F_{0}(u),\phi\right\rangle
$
as $\epsilon\to0$, for each $\phi\in H^{s_{0}}(\Omega)$.
Also, this limit is uniform for $\phi$ in compact sets of $H^{s_{0}}(\Omega)$ since $\{ F_{0,\epsilon}(u)\in H^{-s_{0}}(\Omega): \epsilon\in(0, \epsilon_{0}] \}$ is equicontinuous
for fixed $u\in H^{1}(\Omega)$. Hence, for $1/2<s_{0}<s <1$, we have that the embedding $H^{s}(\Omega)\hookrightarrow H^{s_{0}}(\Omega)$ is compact, and  \begin{eqnarray}
\label{convpont1}
\left\|F_{0,\epsilon}(u)-F_{0}(u)\right\|_{H^{-s}(\Omega)}=
\sup_{\left\|\phi\right\|_{H^{s}(\Omega)}=1}
\left| \left\langle F_{0,\epsilon}(u)-F_{0}(u),\phi\right\rangle \right| \to 0\qquad \mbox{as $\epsilon\to0$}. 
\end{eqnarray}
Now, we show that (\ref{convpont1}) is uniform for $u\in H^{1}(\Omega)$ with $\left\|u\right\|_{H^{1}(\Omega)}\leq R$. First, observe that $F_{0,\epsilon}, F_0:H^{1}(\Omega) \mapsto H^{-s}(\Omega)$ are continuous in $H^{1}(\Omega)$ with the weak topology, for $\epsilon >0$. Indeed, it follows from Lemma~\ref{lem1} and regularity of $f_0$ that there exist $C$, $K>0$, independents of $\epsilon$, and $0\leq \theta(x)\leq 1$ with $x\in \bar{\Omega}$ such that
\begin{eqnarray*}
\left\|F_{0,\epsilon}(u)-F_{0,\epsilon}(v)\right\|_{H^{-s}(\Omega)} & \leq & C \left(\frac{1}{\epsilon}\int_{\omega_{\epsilon}} \left|\partial_u f_0(x,\theta(x)u(x)+(1-\theta(x))v(x))\right|^{2} \left|u(x)-v(x)\right|^{2} dx\right)^{\frac{1}{2}} \\
& \leq & CK \left\|u-v\right\|_{H^{s}(\Omega)}, \quad \forall u, v \in H^{s}(\Omega).
\end{eqnarray*}
Similarly, we get 
$$
\left\|F_{0}(u)-F_{0}(v)\right\|_{H^{-s}(\Omega)}  \leq CK \left\|u-v\right\|_{H^{s}(\Omega)}, \textrm{ for all } u, v \in H^{s}(\Omega).
$$ 
So, since $H^{1}(\Omega)\hookrightarrow H^{s}(\Omega)$ compactly for $s<1$, we have that $F_{0,\epsilon}$ and $F_{0}$ are continuous in $H^{1}(\Omega)$ with weak topology. Hence, $F_{0,\epsilon}$ and $F_{0}$ are uniformly continuous in $\{u\in H^{1}(\Omega) \; : \; \left\|u\right\|_{H^{1}(\Omega)}\leq R\} \subset H^{1}(\Omega)$, proving the result.
\end{proof}

\subsection{Proof of the main result} 

\begin{proof}
Now, we are in a position to prove our main result, Theorem \ref{main}. 
First, we prove that a family of solutions $u^\epsilon$ of \eqref{1} satisfying $\|u^\epsilon\|_{L^\infty(\Omega)} \leq R$ is uniformly bounded in $H^1(\Omega)$ with respect to $\epsilon$.  In fact, we have that $u^\epsilon$ satisfies 
\begin{eqnarray}
\label{formulacao1}
a_\epsilon(u^\epsilon,\varphi) = \left\langle F_{0,\epsilon}(u^\epsilon), \varphi \right\rangle + \left\langle F_{1}(u^\epsilon), \varphi \right\rangle, \quad   \forall \varphi \in H^{1}(\Omega) \quad \textrm{ and } \quad \epsilon>0.
\end{eqnarray}
Hence, if we take $\varphi = u^\epsilon$, we obtain from Lemma \ref{lem1eq} that $|a_\epsilon(u^\epsilon,u^\epsilon)| \leq K \| u^\epsilon \|_{H^{1}(\Omega)}$. Due to Lemma \ref{rem1}, we have that the bilinear form $a_\epsilon$ is uniformly coercive for $\epsilon \geq 0$ (we are using $\lambda > \lambda^*$). Thus, 
$$
C \| u^\epsilon \|_{H^{1}(\Omega)}^2 \leq |a_\epsilon(u^\epsilon,u^\epsilon)| \leq K \| u^\epsilon \|_{H^{1}(\Omega)},
$$
for some constant $C>0$ independent of $\epsilon$, which implies $\| u^\epsilon \|_{H^{1}(\Omega)} \leq K/C$ for all $\epsilon \geq 0$.

Since $u^\epsilon$ is uniformly bounded in $H^1(\Omega)$, we can extract a weakly convergent subsequence, still denoted by $u^\epsilon$,  such that $u^\epsilon \rightharpoonup u$, $w-H^1(\Omega)$, for some $u \in H^1(\Omega)$. It is easy to see that $u$ satisfies our limit problem \eqref{2}. Indeed, we use Lemma \ref{Potential} and Lemma \ref{lem1eq} to pass to the limit in (\ref{formulacao1}) as $\epsilon \to 0$, and obtain 
$$
a_0(u,\varphi) = \left\langle F_0(u), \varphi \right\rangle + \left\langle F_1(u), \varphi \right\rangle,
\textrm{ for all } \varphi \in H^{1}(\Omega).$$
Now, we can prove that $u^\epsilon \to u$ in $H^1(\Omega)$ showing the convergence of the norms. For this, we pass to the limit in 
$
a_\epsilon(u^\epsilon,u^\epsilon) = \left\langle F_{0,\epsilon}(u^\epsilon), u^\epsilon \right\rangle + \left\langle F_{1}(u^\epsilon), u^\epsilon \right\rangle
$ and we use that the norm is lower semicontinuous with respect to the weak convergence (see \cite[Proposition 1.14]{cioranescu}), that is, $\|u\|_{H^1(\Omega)} \leq \displaystyle \liminf_\epsilon \|u^\epsilon\|_{H^1(\Omega)}$, to obtain
\begin{eqnarray*}
\int_\Omega |\nabla u|^2 \, d\xi & \leq & \liminf_\epsilon \int_\Omega |\nabla u^\epsilon|^2 \, d\xi \leq \limsup_\epsilon \int_\Omega |\nabla u^\epsilon|^2 \, d\xi  \\
& \leq & - \int_\Omega \lambda |u|^2 \, d\xi - \int_\Gamma V_0 \, |u|^2 \, d\xi + \int_\Gamma \mu \, f_0(\xi,u) \, u \, d\xi + \int_\Omega f_1(\xi,u) \, u \, d\xi=\int_\Omega |\nabla u|^2 \, d\xi.
\end{eqnarray*}

\end{proof}

\section{Final conclusion}

We have shown that  the steady state solutions  of a homogeneous Neumann problem for a nonlinear reaction diffusion equation converge to a certain limit problem when some reaction and potential terms are concentrated in a small neighborhood of  of the boundary.

 In our analysis we showed  that the family of steady state  solutions  converges in $H^1$-norm to a solution $u$ of an equation of the same type, with a nonlinear boundary condition that captures both the profile and the oscillatory behavior of the
boundary. 

An important feature here is that we are dealing with the case where the boundary  presents a highly oscillatory behavior and, as a consequence, the limit problem is not obvious from the start.  It is  also worth remembering that the domain considered here is only Lipschitz continuous. 

We use some results developed in \cite{arrieta} that describe how different concentrated integrals converge to surface integrals to obtain a rigorous strong convergence result in Theorem \ref{main}.

 A natural question is whether such approximation  results can be improved  
 in order to describe  the asymptotic behavior of the Dynamical System generated by the parabolic equation associated to the problem \eqref{1} posed in more general regions of the $\R^N$. It is our goal to  investigate this question in a forthcoming paper.

\end{document}